\documentclass[12pt,reqno]{amsart}

\usepackage{amsmath,amssymb,amsfonts}
\usepackage[paperwidth=8.26in,paperheight=11.69in,
left=1in, right=1in, bottom=1.1in]{geometry}

\usepackage{xcolor,cancel}

\usepackage{thmtools}
\declaretheoremstyle[
headfont=\normalfont\bfseries,
headindent= 0pt, 
bodyfont=\em,
spaceabove=8pt,
spacebelow=8pt
]{thm}
\declaretheoremstyle[
headfont=\normalfont\em,
headindent= 0pt, 
spaceabove=8pt,
spacebelow=8pt
]{remark}
\declaretheoremstyle[
headfont=\normalfont\bfseries,
headindent= 0pt, 
spaceabove=8pt,
spacebelow=8pt
]{example}
\declaretheoremstyle[
headfont=\normalfont\bfseries,
headindent= 0pt, 
spaceabove=8pt,
spacebelow=8pt
]{definition}
\declaretheorem[name=Theorem,style=thm,numberwithin=section,
]{thm}
\declaretheorem[name=Proposition,style=thm,sibling=thm]{prop}
\declaretheorem[name=Lemma,style=thm,sibling=thm]{lem}

\declaretheorem[name=Definition,style=thm,sibling=thm]{defn}

\usepackage{cleveref}
\crefname{thm}{Theorem}{Theorems}
\crefname{prop}{Proposition}{Propositions}
\crefname{lem}{Lemma}{Lemmas}
\crefname{cor}{Corollary}{Corollaries}
\crefname{example}{Example}{Examples}
\crefname{defn}{Definition}{Definitions}
\crefname{rem}{Remark}{Remarks}

\crefname{enumi}{}{}
\crefname{enumii}{}{}
\crefname{enumiii}{}{}
\crefname{equation}{}{}

\numberwithin{equation}{section}

\newcommand{\ob}{\bar{1}}

\renewcommand{\Re}{\operatorname{Re}}
\renewcommand{\Im}{\operatorname{Im}}

\begin{document}
	
\title[An upper bound for the first positive eigenvalue of Kohn Laplacians]{An upper bound for the first positive eigenvalue of the Kohn Laplacian on Reinhardt real hypersurfaces}

\author{Gian Maria Dall'Ara}
\address{Istituto Nazionale di Alta Matematica ``F. Severi"\\ Research Unit Scuola Normale Superiore\\
	Piazza dei Cavalieri, 7, 56126 Pisa (Italy)}
\email{dallara@altamatematica.it}

\author{Duong Ngoc Son}
\address{Fakult\"at f\"ur Mathematik, Universit\"at Wien, Oskar-Morgenstern-Platz 1, 1090 Wien (Austria)}
\email{son.duong@univie.ac.at}

\thanks{This project was begun while the first-named author was a Marie Sk\l odowska-Curie Research Fellow at the University of Birmingham. He gratefully
	acknowledges the support of the European Commission via the Marie Sk\l odowska-Curie Individual Fellowship ``Harmonic Analysis on Real Hypersurfaces in Complex Space" (ID 841094). The second-named author was supported by the Austrian Science Fund (FWF): Projekt I4557-N}

\begin{abstract}
A real hypersurface in $\mathbb{C}^2$ is said to be Reinhardt if it is invariant under the standard $\mathbb{T}^2$-action on $\mathbb{C}^2$. Its CR geometry can be described in terms of the curvature function of its ``generating curve'', i.e., the logarithmic image of the hypersurface in the plane $\mathbb{R}^2$. We give a sharp upper bound for the first positive eigenvalue of the Kohn Laplacian associated to a natural pseudohermitian structure on a compact and strictly pseudoconvex Reinhardt real hypersurface having closed generating curve (which amounts to the $\mathbb{T}^2$-action being free). Our bound is expressed in terms of the $L^2$-norm of the curvature function of the generating curve and is attained if and only if the curve is a circle.
\end{abstract}

\date{Oct 13, 2021}
\maketitle

\section{Introduction}

Let $(M,\theta)$ be a compact strictly pseudoconvex pseudohermitian manifold. The associated Kohn Laplacian $\Box_b$ is a nonnegative self-adjoint second-order operator whose spectrum $\mathrm{Spec}(\Box_b)$ in $(0,+\infty)$ consists of discrete eigenvalues of finite multiplicity, while zero is an isolated eigenvalue if and only if $M$ is embeddable into some complex space. The spectrum $\mathrm{Spec}(\Box_b)$ in general, and estimates for the first positive eigenvalue (in the embeddable case) in particular, have been studied in several papers. For example, in \cite{chanillo2012embeddability} a Lichnerowicz-type estimate for the first positive eigenvalue in terms of the Webster scalar curvature (or Ricci curvature) was established, while the characterization of the equality case in this estimate was treated in \cite{li2015new} and \cite{case2021lichnerowicz}, for the higher-dimensional and the three-dimensional cases respectively. In \cite{li2018sharp}, several upper bounds for the first positive eigenvalue of the Kohn Laplacian on the boundaries of certain domains in $\mathbb{C}^2$ were obtained; the bounds are extrinsic, depending on the realization of the CR manifolds as embedded compact real hypersurfaces in $\mathbb{C}^2$. All of the aforementioned bounds are sharp, with equality holding if $(M,\theta)$ is isomorphic to the standard CR sphere. \medskip 

The main purpose of this paper is to provide an upper estimate for the first positive eigenvalue of the Kohn Laplacian associated to a natural choice of pseudohermitian structure on Reinhardt real hypersurfaces in $\mathbb{C}^2$. A real hypersurface $M\subset \mathbb{C}^2_{z,w}$ is said to be Reinhardt if it is invariant under the standard action of the $2$-dimensional torus $\mathbb{T}^2$ on $\mathbb{C}^2$ (here $\mathbb{T}:=\mathbb{R}/2\pi\mathbb{Z}$). Under this assumption, the open subset $M^{\ast}: = \{(z,w) \in M \mid zw \ne 0\}$ can be recovered from its image $\gamma\subset\mathbb{R}^2$ under the mapping $(z,w)\mapsto (\xi, \eta)=(\log|z|, \log|w|)$. We call $\gamma$ the ``generating curve" of $M$. (It may be of interest to remark that, by a result of Sunada \cite{sunada1978holomorphic}, two bounded Reinhardt domains in $\mathbb{C}^2$ are biholomorphically equivalent if and only if the generating curves of their boundaries are rigid equivalent.)

If $\gamma(s) = (\xi(s),\eta(s))$ is an arc-length parametrization of the generating curve $\gamma$ (the symbol $s$ will always denote the arc-length parameter), then $M^{\ast}$ is parametrized by
\begin{equation}\label{parametrization}
	(s, x ,y) \mapsto (e^{\xi(s)+ix }, e^{\eta(s) + iy}) \in \mathbb{C}^2,
\end{equation} 
	where $i = \sqrt{-1}$ and $x,y\in \mathbb{T}$. Notice that $M^\ast$ is diffeomorphic to $\gamma\times \mathbb{T}^2$, and that $M=M^\ast$ if and only if the induced $\mathbb{T}^2$-action on $M$ is free. In this paper, we work under the following assumptions on the Reinhardt hypersurface $M\subset\mathbb{C}^2$:\begin{equation}\label{assumptions}
	M \ \text{is compact, connected, and}\ M=M^\ast.
	\end{equation} 
In this case, the generating curve is a simple closed curve, and hence $M$ is diffeomorphic to a $3$-dimensional torus.
	
\begin{defn}\label{closed_defn} We say that a Reinhardt hypersurface $M$ satisfying \eqref{assumptions} has a ``closed generating curve".	
\end{defn}

So, although the sphere is Reinhardt, the hypersurfaces of interest here are \textit{not} diffeomorphic to the sphere. By \cite{li2015new} and \cite{case2021lichnerowicz}, the aforementioned estimates for the first eigenvalue of the Kohn Laplacian are not sharp for manifolds satisfying Definition \ref{closed_defn}, and this motivates the present paper.\medskip

Let $M$ be a strictly pseudoconvex Reinhardt hypersurface with closed generating curve $\gamma(s)=(\xi(s), \eta(s))$, and let $\kappa = \kappa(s)$ be the curvature function of $\gamma$. Using the parametrization \eqref{parametrization}, the function $\kappa$ may be lifted to a $\mathbb{T}^2$-invariant function on $M$. Since $M$ is strictly pseudoconvex, we can choose an orientation of $\gamma$ such that
$\kappa(s) > 0$. A natural choice of a pseudohermitian structure $\theta$ is then (following \cite{burns1988global})
\begin{equation} \label{e:12}
\theta = \eta{'} \, dx -\xi{'} dy, 
\end{equation} 
whose associated volume form satisfies $dV:=\theta \wedge d\theta = \kappa(s) ds\wedge dx \wedge dy$. We call \eqref{e:12} the ``normalized pseudohermitian structure" on $M$. As is well-known, $\theta$ allows to define an associated Kohn Laplacian $\Box_b$ (acting on functions). Denote by $\lambda_1(\Box_b)$ its first positive eigenvalue, which is well-defined since $M$ is strictly pseudoconvex and embedded (cf. \cite{kohn1986range}).

Our main result is
\begin{thm}\label{thm:1}
	Let $M\subset \mathbb{C}^2$ be a strictly pseudoconvex Reinhardt hypersurface with closed generating curve	$\gamma$. Let $\kappa(s) > 0$, $0\leq s < \ell$, be the curvature of $\gamma$ in an arc-length parametrization $s$, where $\ell$ is the length of $\gamma$. Let $\lambda_1(\Box_b)$ be the first positive eigenvalue of the Kohn Laplacian associated to the normalized pseudohermitian structure $\theta$ defined in \eqref{e:12}. Then
	\begin{equation}\label{e:13}
	\lambda_1(\Box_b) \leq \frac{1}{4\pi} \int_0^{\ell} \kappa^2 ds.
	\end{equation} 
	Equality occurs if and only if $\gamma$ is a circle, namely, if and only if
	\[M:=\left\{\left(\log a|z|\right)^2+\left(\log b|z|\right )^2=\frac{\ell^2}{4\pi^2}\right\}\] for some $a,b>0$.
\end{thm}

Since $\gamma$ is a simple closed curve, by the Hopf winding number theorem,
	\[
		\mathrm{Vol}(M) := \int_M \theta \wedge d\theta
		= 
		4\pi^2 \int_0^{\ell} \kappa(s) ds 
		=
		8\pi^3.
	\]
Let $R$ be the Webster scalar curvature of $\theta$. It was shown in \cite{burns1988global} that $R = \kappa /2 - (\log \kappa)''/(2\kappa)$ and thus
\[
\int_M R \, \theta \wedge d\theta 
=
4\pi^2 \int_0^{\ell} \left(\kappa /2 - (\log \kappa)''/(2\kappa)\right) \kappa ds
=
2\pi^2 \int_0^{\ell} \kappa^2 ds.
\]
Estimate \eqref{e:13} is then equivalent to
\[
	\lambda_1(\Box_b) \leq \frac{1}{8\pi^3} \int _M R \theta \wedge d\theta = \frac{1}{\mathrm{Vol}(M)} \int _M R \theta \wedge d\theta.
\]
In other words, $\lambda_1(\Box_b)$ is not larger than the average value of the Webster scalar curvature on~$M$. For a comparison, the Chanillo-Chiu-Yang's lower bound \cite{chanillo2012embeddability} yields
\[
	\lambda_1 \geq \frac12 \min R,
\]
where the equality holds if only if $M$ is the sphere (which is never the case in our setting, since $M$ is diffeomorphic to a 3-torus).

\medskip 

Let us conclude this introduction by remarking that our analysis uses the fact that the generating curve is closed in an essential way. In fact, this assumption allows to reduce the analysis of Kohn Laplacians on this class of real hypersurfaces to that of a family of periodic differential equations. In particular, the Kohn Laplacian on the ``symmetric model'' (i.e., when $\gamma$ is a circle) can be decomposed into a family of ``quasi-exactly solvable'' periodic Schr\"odinger equations, the Whittaker--Hill equations, that have been studied for a long time (cf. \cite{turbiner2016one,volkmer2008approximation}).

\section{Analysis on Reinhardt real hypersurfaces}

As in the introduction, we have the following data:\begin{itemize}
	\item a strictly pseudoconvex Reinhardt real hypersurface $M\subset\mathbb{C}^2$ with closed generating curve.
	\item its generating curve $\gamma(s)=(\xi(s), \eta(s))$, parametrized by arc-length $s$, so that $\mathbb{R}/\ell\mathbb{Z}\times \mathbb{T}^2\simeq M$ via \eqref{parametrization}. Here $\ell$ is the length of $\gamma$, and we may naturally think of $s$ as a variable in $\mathbb{R}/\ell\mathbb{Z}$. 
	\item $\kappa = \kappa(s)$, the positive curvature function of $\gamma$, also expressed with respect to the arc-length parameter and thought of as a $\mathbb{T}^2$-invariant function $\kappa(s,x,y)=\kappa(s)$ on $M$. 
	\item the normalized pseudohermitian structure 	$\theta = \eta{'} \, dx -\xi{'} dy$, with attached volume form $dV=\kappa(s)ds\wedge dx\wedge dy$. 
	\item the Kohn Laplacian $\Box_b$ associated with $\theta$, with first positive eigenvalue $\lambda_1(\Box_b)$. 
	\end{itemize}

To obtain a more concrete description of $\Box_b$, we exploit the $\mathbb{T}^2$-invariance of our setting, that is, Fourier analysis. First of all, a function $u$ on $M$ can be identified with a triple periodic function of period $(\ell, 2\pi, 2\pi)$, i.e., as a function on $\mathbb{R}/\ell \mathbb{Z}\times \mathbb{T}^2$. Therefore, $u$ admits a partial Fourier expansion
\[
u(s, x, y) 
=
\sum_{m,l \in \mathbb{Z}} u_{m,l}(s) \exp(imx + ily),
\]
where \begin{equation}\label{fourier}
u_{m,l}(s)=\frac{1}{4\pi^2}\int_{\mathbb{T}^2}u(s,x,y)e^{-imx-ly}dx\, dy
\end{equation} 
are the Fourier coefficient functions, which are $\ell$-periodic, and we have the orthogonal decomposition
\[
	L^2(M,dV) = \bigoplus_{l,m \in \mathbb{Z}} H_{m,l},
\]
where $H_{m,l} = \{v(s) \exp(imx + ily)\colon \ v\in L^2(\mathbb{R}/\ell\mathbb{Z}, \kappa(s)ds)\}$. Notice that $L^2(\mathbb{R}/\ell\mathbb{Z},\kappa(s)ds)$ coincides with $L^2(\mathbb{R}/\ell\mathbb{Z},ds)$ as a set, but the scalar products are different. 

\begin{prop}\label{prop:21} Let $M$ be a strictly pseudoconvex Reinhardt real hypersurface in $\mathbb{C}^2$ with closed generating curve, and let $\Box_b$ be the Kohn Laplacian associated to the normalized pseudohermitian structure. \newline Then the spaces $H_{m,l}$ are invariant under $\Box_b$ for every $l, m\in \mathbb{Z}$. More precisely, if $u(x, y, s) = v(s)\, e^{imx + ily}$ is smooth, then
\begin{align}\label{e:boxform}
\Box_ b u
=
\left\{\frac{1}{2\kappa}\left(- \frac{d^2}{ds^2} + V_{m,l}\right)v\right\}e^{imx + ily} .
\end{align}
Here, \[V_{m,l} = (lp + mq)^2 + \kappa (lq - mp),\] where $q(s) = \xi'(s)$, $p(s) = \eta'(s)$, and $\kappa(s)$ is the curvature function of $\gamma$.\newline
In particular, $\Box_bu=0$ (for $u$ as above) if and only if \begin{equation}\label{kerB}
	v(s) = C \exp (l\eta(s) + m \xi(s)),
\end{equation} 
where $C\in \mathbb{C}$ is a constant, that is, if and only if $u(s,x,y)=C \exp\left(m(\xi + i x) + l(\eta + iy)\right)$.
\end{prop}

Observe that the expression in curly brackets in formula \eqref{e:boxform} defines a 1-dimensional Schr\"odinger operator on $\mathbb{R}/\ell\mathbb{Z}$ with potential $V_{m,l}(s) = (lp + mq)^2 + \kappa (lq - mp)$. Thus, the analysis of $\Box_b$ on $L^2(M, dV)$ is reduced to that of the two-parameter family of operators $\left\{\frac{1}{2\kappa}\left(- \frac{d^2}{ds^2} + V_{m,l}\right)\right\}_{m,l}$ on $L^2(\mathbb{R}/\ell\mathbb{Z}, \kappa(s)ds)$. 

We also remark that, by \eqref{kerB}, $u(s,x,y)=v(s)e^{ilx+imy}$ is CR if and only if it is the restriction to $M$ of a multiple of the holomorphic monomial $z^m w^l$. 

\begin{proof} We choose a frame $Z_1$ of $T^{(1,0)}M$ as follows (cf. \cite{burns1988global}). Since $Z_1$ must satisfy
\[
Z_1 \left( e^{\xi(s) - i x}\right) 
=
Z_1 \left( e^{\eta(s) - i y}\right) = 0,
\] 
we find that
\[
	Z_1
	=
	\phi\left(-i\, \eta{'} \, \frac{\partial}{\partial y} - i\, \xi{'} \, \frac{\partial}{\partial x }+\frac{\partial}{\partial s}\right),
\]
for some function $\phi$ on $M$. Thus, if we require that the Levi ``matrix'' is the identity, i.e., $h_{1\ob} = -i d\theta (Z_1, Z_{\ob}) = 1$, then we must have $2\kappa |\phi|^2 = 1$. Hence, we take $\phi = 1/\sqrt{2\kappa}$. With this choice of $Z_1$, we have the dual coframe 
\[
\theta^1 := \sqrt{\kappa/2}\,\left(i\,\xi{'} \, dx + i\, \eta{'} \, dy+ds \right).
\]
A computation (cf. \cite{burns1988global}) shows that the connection and torsion forms are:
\begin{align}\label{e:cf}
\omega_{1}{}^{1}
&=
\kappa'/\sqrt{8\kappa^3}\,(\theta^1 - \theta^{\ob}) - i(\kappa /2) \theta\\
\tau^1 
& = \frac{i\kappa}{2}\theta^{\ob} \label{e:torsion}.
\end{align} Thus, \begin{align}
d\omega_1{}^{1} 
&=
\left[\frac{\kappa}{2}-\frac{(\log \kappa)''}{2\kappa}\right]\theta^1\wedge \theta^{\ob} - \frac{i\, \kappa'}{\sqrt{2\kappa}}[\theta^1\wedge \theta + \theta^{\ob} \wedge \theta]. \label{e:dcf}
\end{align}
%
%Notice that Equations \cref{e:cf,e:dcf} correct minor mistakes in two formulas appeared in \cite[page 344]{burns1988global}. 

The Reeb field $T$ is determined by $d\theta(T,Z_1) = 0$, $\theta(T) = 1$, and its reality. Thus
\begin{align}
T =
\eta{'} \,\frac{\partial}{\partial x } - \xi{'} \, \frac{\partial}{\partial y}.
\end{align}
If $u$ is a smooth function, then we have
\begin{equation} 
	d u= (Z_1 u) \theta^1 + (Z_{\ob} u)\theta^{\ob} + (Tu) \theta
\end{equation} 
and $\partial_b u = (Z_1 u) \theta^{1}$, which is a genuine one-form on $M$. Integration by parts implies that 
\begin{equation} 
	\Box_b u = - u_{\ob,}{}^{\ob}.
\end{equation} 
Here, the indices preceded by a comma indicate the covariant derivative:
\begin{equation} 
	u_{\ob, 1} = Z_1 Z_{\ob} u - \omega_{\ob}^{\ob}(Z_1) Z_{\ob} u.
\end{equation}
In the chosen frame, raising and lowering indices of tensors only change their ``types'' since $h_{1\ob} =1$.
Since $u(x, y, s) = v(s)\, \exp(im x + i l y)$, where $l$ and $m$ are constants, we have
\[
u_{\bar{1}} = Z_{\ob} u =  \frac{1}{\sqrt{2 \kappa }} e^{i (l y+m x)} \left(v' - (l p + m q) v\right).
\]
Since $\Box_bu=0$ if and only if $Z_{\ob} u=0$, identity \eqref{kerB} follows immediately. Moreover,
\begin{align*}
- \Box_b u 
& = Z_1(u_{\bar{1}}) - \omega_{\ob}^{\ob}(Z_1) u_{\bar{1}} \\
& = Z_1\left(\frac{1}{\sqrt{2\kappa }} e^{i (l y+m x)} \left(v' - (l p + m q) v\right)\right) \\
& \quad -\frac{1}{4\kappa } e^{i (l y+m x)} (\log \kappa)' \left(v' - (l p + m q) v\right)\\
& = \frac{1}{2\kappa } e^{i(l y+m x)} \left(v'' - (l^2 p^2+2 l m p q+l 
p'+m^2 q^2+m q')v\right).
\end{align*}
Substituting $p' = \kappa q$ and $q' = - \kappa p$, we obtain
\eqref{e:boxform}. 
\end{proof}

Proposition \ref{prop:21} can be used to verify that $\Box_b$, initially defined as a symmetric operator on $C^\infty(M)\subset L^2(M,dV)$, is essentially self-adjoint. One way to see this is the following. By a well-known criterion \cite[Corollary to Theorem VIII.3]{reed-simon-I}, it is enough to check that $\Box_b^*-\overline{\lambda}$ is injective when $\lambda\in \mathbb{C}\setminus \mathbb{R}$. Here $\Box_b^*$ denotes the Hilbert space adjoint. In other words, we have to show that a function $f\in L^2(M,dV)$ such that 
\begin{equation*}
	\int_M f\, \overline{\left(\Box_b-\lambda\right)\varphi} \,dV = 0   \qquad\forall \varphi\in C^\infty(M),
\end{equation*} 
must necessarily vanish. The choice $\varphi(s,x,y)=v(s)e^{imx+ily}$ for $v\in C^\infty(\mathbb{R}/\ell\mathbb{Z})$ and $l,m\in \mathbb{Z}$ gives
\[\begin{split}
	0 
	&= \int_M f(s,x,y)\, \overline{\left(\left\{\frac{1}{2\kappa}\left(- \frac{d^2}{ds^2} + V_{m,l}\right)v\right\}-\lambda\right)v(s)} \,\kappa(s)ds\, e^{-imx-ily}dx \,dy \\
&= \int_{\mathbb{R}/\ell\mathbb{Z}} f_{ml}(s)\, \overline{\left(\left\{\frac{1}{2\kappa}\left(- \frac{d^2}{ds^2} + V_{m,l}\right)v\right\}-\lambda\right)v(s)} \,\kappa(s)ds , 
\end{split}\]
where the functions $f_{ml}\in L^2(\mathbb{R}/\ell\mathbb{Z}, \kappa(s)ds)$ are as in \eqref{fourier}. In other words, $(B_{ml}^*-\overline\lambda)f_{ml}=0$, where \[B_{ml} =	\frac{1}{2\kappa (s)} \left(-\frac{d^2}{ds^2} + V_{m,l}(s)\right), \]
initially defined on $C^\infty(\mathbb{R}/\ell\mathbb{Z})$ and symmetric with respect to the $L^2(\mathbb{R}/\ell\mathbb{Z}, \kappa(s)ds)$ scalar product. This operator is essentially self-adjoint (this is classical, see e.g. \cite[Theorem X.28]{reed-simon-II}), and therefore, appealing again to the essential self-adjointness criterion cited above, we must have $f_{ml}=0$ for every $m,l$, and hence $f=0$, as we wanted.

It is also classical that $L^2(\mathbb{R}/\ell\mathbb{Z}, \kappa(s)ds)$ admits an orthonormal basis consisting of smooth eigenfunctions $\psi^{ml}_n$ ($n=0,1, \ldots$) of $B_{ml}$, corresponding to a sequence of eigenvalues \[
	0=\lambda_0(B_{ml})<\lambda_1(B_{ml})\leq \cdots
\] tending to $+\infty$.  Notice that $\lambda_0(B_{m,l})$ vanishes and has multiplicity one, because $B_{lm}$ is nonnegative (as a consequence of the nonnegativity of $\Box_b$) and its kernel is one-dimensional.

Then, by Proposition \ref{prop:21}, the collection $\left\{\psi^{ml}_n(s)e^{imx+ily}\right\}_{m,l,n}$ is an orthonormal basis of $L^2(M,dV)$ consisting of eigenfunctions of $\Box_b$ (they are certainly in the domain of $\Box_b$, since they are smooth). The following two observations follow (for $\mu>0$): \begin{enumerate}
\item To show that $\lambda_1(\Box_b)\leq \mu$, it is enough to show that $\lambda_1(B_{ml})\leq \mu$ for at least one pair $(m,l)\in \mathbb{Z}^2$. 
\item To prove that $\lambda_1(\Box_b)\geq \mu$, we need to ensure that for every pair $(m,l)\in \mathbb{Z}^2$, we have the inequality $\lambda_1(B_{ml})\geq \mu$. 
\end{enumerate}

The first positive eigenvalues of the $B_{m,l}$'s are uniformly bounded below by a positive constant, namely,
	\begin{equation} 
		\lambda_1(B_{ml}) \geq \lambda_1(\Box_b) > 0.
	\end{equation} 
This uniform estimate seems to be nonobvious from the point of view of Schr\"odinger operators. More precisely, since the Webster scalar curvature is $R = (1/2)\left(\kappa - \kappa^{-1}(\log \kappa)''\right)$, the Lichnerowicz-type estimate for $\lambda_1(\Box_b)$ of \cite{chanillo2012embeddability}  gives 
\begin{equation} 
	\lambda_1(B_{ml}) \geq \frac{1}{4} \min_{s\in \mathbb{R}/\ell\mathbb{Z}} \left(\kappa(s) - \frac{(\log \kappa(s))''}{\kappa(s)}\right) \quad \forall m,l.
\end{equation} 
Unfortunately, this estimate is far from being sharp even when $\kappa$ is constant.

\section{The upper bound for $\lambda_1(\Box_b)$}

In this section, we prove inequality \eqref{e:13} of Theorem \ref{thm:1}. By the discussion in the last section, this will follow from the next proposition. 

\begin{prop} Let $\gamma:\mathbb{R}/\ell\mathbb{Z}\rightarrow \mathbb{R}^2$ ($s:=$ arc-length parametrization) be a smooth simple closed curve, and let $\kappa(s)$ be its curvature function. Assume that $\kappa(s)>0$ for every $s$. Let $B:=B_{00}=-(2\kappa)^{-1} d^2/ds^2$, which defines a nonnegative and essentially self-adjoint operator on $L^2(\mathbb{R}/\ell\mathbb{Z}, \kappa ds)$. Then its first positive eigenvalue satisfies the bound:
\[
	\lambda_1(B) \leq \frac{1}{4\pi} \int_{\mathbb{R}/\ell\mathbb{Z}} \kappa^2 ds,
\]
with equality if and only if $\kappa$ is constant, that is, if and only if $\gamma$ is a circle of radius $\frac{1}{2\lambda_1(B)}$.
\end{prop}

Notice that it also follows from this statement that, if equality is attained in \eqref{e:13}, then the generating curve is a circle of radius $\frac{1}{2\lambda_1(\Box_b)}$. 

\begin{proof} Denote by $(u,v)$ the inner product in $L^2(\mathbb{R}/\ell\mathbb{Z}, \kappa ds)$, i.e.,
	\begin{equation} 
		(u,v) := \int_{\mathbb{R}/\ell\mathbb{Z}} u(s)\overline{v(s)}\kappa(s)\, ds.
	\end{equation} Notice that
\[
	(B u, v)
	=
	-\frac12 \int_{\mathbb{R}/\ell\mathbb{Z}} \kappa^{-1} u''(s) \overline{v}(s) \kappa ds
	=
	\frac12 \int_{\mathbb{R}/\ell\mathbb{Z}} u'(s) \overline{v'(s)} ds,
\]
The first positive eigenvalue of $B$ has the following variational characterization:
\begin{equation} \label{variational}
	\lambda_1:=\lambda_1(B) = \inf \left\{\frac{(B v, v)}{(v,v)} \colon v \in W^{1,2} (\mathbb{T},\mathbb{R}), \int_0^\ell  v \kappa  ds = 0, v \ne \, \text{constant}\right\}.
\end{equation} 
Recall the notation $\gamma(s)=(\xi(s), \eta(s))$ and $\gamma'(s)=(q(s), p(s))$. We have $\int_{\mathbb{R}/\ell\mathbb{Z}} p\kappa ds = \int_{\mathbb{R}/\ell\mathbb{Z}} q'\, ds =0$ and $\int_{\mathbb{R}/\ell\mathbb{Z}} q\kappa ds = -\int_{\mathbb{R}/\ell\mathbb{Z}} p'\, ds =0$. Moreover, 
\begin{align*}
	(Bp, p)
	&=
	-\frac{1}{2}\int_{\mathbb{R}/\ell\mathbb{Z}} \kappa^{-1} p'' p\, \kappa ds
	=
	-\frac{1}{2} \int_{\mathbb{R}/\ell\mathbb{Z}} p'' p\, ds \notag \\
	& =
	\frac{1}{2} \int_{\mathbb{R}/\ell\mathbb{Z}} (p')^2 ds \notag 
	=
	\frac{1}{2} \int_{\mathbb{R}/\ell\mathbb{Z}} \kappa^2 q^2 ds.
\end{align*}
Similarly, we have 
\begin{equation*} 
	(Bq, q)
	=
	\frac{1}{2} \int_{\mathbb{R}/\ell\mathbb{Z}} \kappa^2 q^2 ds,
\end{equation*} 
and thus 
\begin{equation*} 
	(Bp, p) + (Bq,q) = \frac{1}{2} \int_{\mathbb{R}/\ell\mathbb{Z}} \kappa^2 ds,
\end{equation*} 
since $p^2 + q^2 = 1$. On the other hand, by the Hopf winding number theorem,
\begin{equation*} 
	(p,p) + (q,q) = \int_{\mathbb{R}/\ell\mathbb{Z}} (p^2 + q^2)\, \kappa ds= \int_{\mathbb{R}/\ell\mathbb{Z}} \kappa ds = 2\pi.
\end{equation*} 
Hence, by the variational characterization \eqref{variational}, we have that
\begin{equation}\label{upperbound} 
	\lambda_1 \leq \frac{1}{4\pi} \int_{\mathbb{R}/\ell\mathbb{Z}} \kappa^2 ds.
\end{equation} 

Suppose that equality occurs in \eqref{upperbound}. Then necessarily $B p = \lambda_1 p$ and $B q = \lambda_1q$, 
\begin{equation} 
	\lambda_1 p = B p = -(2\kappa)^{-1} p''
	=
	(\kappa/2) p - \frac{1}{2} (\log \kappa)' q,
\end{equation} 
and thus
\begin{equation} 
	\lambda_1 p^2
	=
	(\kappa/2) p^2 - \frac{1}{2} (\log \kappa)' pq.
\end{equation} 
Similarly, 
\begin{equation} 
\lambda_1 q^2 = (\kappa/2) q^2 + \frac{1}{2} (\log \kappa)' pq.
\end{equation} 
Adding these together and using $p^2 + q^2 = 1$, we obtain $\lambda_1 \equiv \frac{\kappa}{2}$,
i.e., $\kappa\equiv 2\lambda_1$ is constant and $\gamma$ is a circle of radius $1/(2\lambda_1)$. This completes the proof.
\end{proof}

\section{Equality case}

In this section we study the spectrum of $\Box_b$ (defined with respect to the normalized pseudohermitian structure \eqref{e:12}) on the Reinhardt real hypersurface having as closed generating curve $\gamma$ a circle of radius $\kappa^{-1}$, that is, we assume that the curvature function $\kappa(s) = \kappa$ is constant. In particular, we show that in this case $\kappa/2$ is the \textit{lowest} positive eigenvalue of $\Box_b$, thus completing the proof of our main theorem. 

\noindent
\textit{Remark:} Since in this case the Webster scalar curvature is $R\equiv \kappa/2$, the inequality of Chanillo-Chiu-Yang \cite{chanillo2012embeddability} already gives the lower bound $\lambda_1 \geq \kappa/4$. In fact, the needed assumption that the CR Paneitz operator is nonnegative has been shown to hold for embeddable manifolds by Takeuchi \cite{takeuchi}.

Let then $\gamma(s)=(\xi(s), \eta(s))$, where $\xi(s) = \kappa^{-1}\cos (\kappa s)$, $\eta(s) = \kappa^{-1}\sin (\kappa s)$, $q(s) = -\sin (\kappa s)$, $p(s) = \cos (\kappa s)$, and $s \in \mathbb{R}/(2\pi \kappa^{-1}\mathbb{Z})$. We note in passing that different values of $\kappa>0$ correspond to inequivalent CR manifolds: see \cite{burns1988global}.

Recall that our task boils down to proving that $\lambda_1(B_{ml})\geq \frac{\kappa}{2}$ for every $(m,l)\in \mathbb{Z}^2$, where \[
B_{ml}=\frac{1}{2\kappa}\left(- \frac{d^2}{ds^2} + V_{m,l}\right), \qquad V_{m,l} = (lp + mq)^2 + \kappa (lq - mp).\]
Thus, if we make the change of variables $\tau = \kappa s + \alpha \in \mathbb{T}=\mathbb{R}/2\pi\mathbb{Z}$, $E =  2\kappa^{-1} \lambda$, and $a= \kappa^{-1} \sqrt{m^2+l^2}$, then the eigenvalue equation $B_{ml}v=\lambda v$ becomes the well-known Whittaker--Hill equation 
\begin{equation}\label{e:wh2}
	\mathrm{WH}_au(\tau):=-\frac{d^2 u}{d \tau^2 }+ \left(a^2 \sin^2(\tau) + a \cos(\tau)\right) u  = E \, u.
\end{equation} 
See, e.g., section 7.4 of \cite{magnus1966hill} or \cite{volkmer2008approximation}. Writing $u(\tau) = w(\tau) \exp(- a\cos(\tau))$, \eqref{e:wh2} becomes
\begin{equation}\label{e:Ince}
	\mathrm{I}_aw(\tau):= -\frac{d^2 w}{d \tau^2 } - 2a \sin (\tau )\, \frac{d w}{d \tau}(\tau ) = E w(\tau ),
	\qquad \tau \in \mathbb{R}/2\pi \mathbb{Z}.
\end{equation}

Notice that, while $\mathrm{WH}_a$ is self-adjoint with respect to the $L^2(\mathbb{T}, d\tau)$ scalar product (as a consequence of the self-adjointness of $B_{ml}$ on $L^2(\mathbb{T}, \kappa d\tau)$), the operator $\mathrm{I}_a$ is self-adjoint with respect to $L^2(\mathbb{T}, e^{-2a\cos(\tau)}d\tau)$. Hence, in passing from \eqref{e:wh2} to \eqref{e:Ince}, we simplified the coefficients (passing from a quadratic to a linear trigonometric polynomial) at the expense of losing self-adjointness with respect to $L^2(\mathbb{T}, d\tau)$. 

At any rate, both $\mathrm{WH}_a$ and $\mathrm{I}_a$ have zero as a simple (and smallest) eigenvalue, and preserve the parity of functions (as does the transformation $w\mapsto we^{-a\cos(\tau)}$). Hence every positive eigenvalue has multiplicity two and the corresponding eigenspace is spanned by an odd and an even $2\pi$-periodic function. By restricting $\mathrm{I}_a$ to the Hilbert space of $2\pi$-periodic \textit{odd} functions, which is closure of the linear span of $\{\sin(k\tau)\}_{k\geq 1}$, a simple computation shows that the double eigenvalues of $\mathrm{I}_a$ coincide with the simple eigenvalues of the infinite tridiagonal matrix $\mathcal{I}(a) = \{\mathcal{I}_{kk'}(a)\}_{k,k'\geq 1}$ given by 
\begin{eqnarray*}
	\mathcal{I}_{kk'}(a) = \begin{cases} 
		k^2\quad &k'=k\\
		(k'-k)k'a\quad &|k-k'|=1\\
		0\quad &\text{otherwise}
	\end{cases}
\end{eqnarray*}

Notice that $\mathcal{I}(a)$ is not self-adjoint, because $\{\sin(k\tau)\}_{k\geq 1}$ is not orthogonal with respect to the $L^2(\mathbb{T}, e^{-2a\cos(\tau)})$ scalar product. Denote by $\mathcal{I}_N(a)$ the $N$-th principal minor of $\mathcal{I}(a)$. For example,
\begin{equation} 
\mathcal{I}_4(a)=\begin{pmatrix}
		1 & 2 a & 0 & 0 \\
		-a & 4 & 3 a & 0 \\
		0 & -2 a & 9 & 4 a \\
		0 & 0 & -3 a & 16 \\
	\end{pmatrix}
\end{equation} 

See \cite{volkmer2008approximation} for a thoroughly discussion of this truncation. By the above discussion, to established the desired inequality $\lambda_1(B_{ml})\geq \kappa/2$, we have to show that if $E$ is an eigenvalue of $\mathcal{I}(a)$, then necessarily $E\geq 1$, for every value of $a$. 

In order to achieve that, we need the following two lemmas. The first follows from known results on Jacobi matrices. 

\begin{lem}\label{lem:jacobi_spectrum}
	Let $\mathcal{I}(a)$ and $\mathcal{I}_N(a)$ be the matrices defined above. Then every eigenvalue of $\mathcal{I}(a)$ is a limit of eigenvalues of $\mathcal{I}_N(a)$ as $N$ tends to $+\infty$.
\end{lem}
\begin{proof}
	The result for complex Jacobi matrices is by now classical. For the non-symmetric case, this follows from \cite[Theorem~2]{volkmer2008approximation} or \cite[Theorem~2.1]{malejki2018eigenvalues}.
\end{proof}

\begin{lem}\label{lem:tridiag_matrix}
	Assume that the real tri-diagonal matrix
	\[
	A=\begin{pmatrix}
		\delta_1 & \mu_1        &                 & &  \\
		\nu_1	     & \delta_2 & \mu_2       & & \\
		& \nu_2	       & \delta_3 & \ddots & \\
		&                 &    \ddots    & \ddots& \mu_{N-1}\\
		&   &  & \nu_{N-1}& \delta_N \\
	\end{pmatrix}
	\]
	satisfies the following properties:
	\begin{enumerate}
		\item $\delta_k>0$ for every $k$;
		\item $\mu_k\nu_k\leq 0$ for every $k$.
	\end{enumerate}
	Let $0<\delta\leq \frac{\pi}{2}\left(\sum_k\delta_k^{-1}\right)^{-1}$, and $\mu:=\min_k\delta_k$. Then $A$ has no eigenvalues in the open sectorial region
	\[
		T:=\left\{z\in \mathbb{C} \colon -\infty<\Re(z)<\mu,\ |\Im(z)|<\delta(1-\mu^{-1}\Re(z))\right\}.
	\]
\end{lem}

\begin{proof} Assume that $N\geq 2$ to avoid trivialities. Denote by $S(\theta)$ ($\theta\in(0,\pi]$) the sector consisting of nonzero complex numbers whose argument has modulus $<\theta$ (use the principal determination of the argument). Let $\theta_k\in(0,\pi/2)$ be the unique angle such that \[
	\tan \theta_k=\frac{\delta}{\delta_k}.\]
	By the convexity of the tangent function, we have \begin{equation}\label{sumtheta}
		\sum_k\theta_k\leq \sum_k\tan\theta_k = \delta\sum_k\delta_k^{-1}\leq\frac{\pi}{2},
	\end{equation}
	by our choice of $\delta$. 
	
	Let now $A_k$ be the $k$-th principal minor of $A$, and let $P_k(z)=\det(A_k-z\mathbb{I}_k)$ be its characteristic polynomial ($\mathbb{I}_k$ is the $k\times k$ identity matrix). We are going to show that \begin{equation}\label{claim}
		P_k(T)\subset S(\theta_1+\cdots+\theta_k)\qquad \forall k\leq N. 
	\end{equation}
	The case $k=N$ says in particular that the characteristic polynomial of $A$ has no zeros in $T$, as we wanted. 
	
	We are going to prove \eqref{claim} by induction. We use the following elementary geometric facts: 
	\begin{itemize}
		\item The image of $T$ under the map $z\mapsto \delta_k-z$ is contained in the sector $S(\theta_k)$. 
		\item Any sector $S(\theta)$ ($\theta\leq \pi$) is invariant under right translations $z\mapsto z+b$, where $b\geq 0$, and dilations $z\mapsto az$, where $a>0$.
		\item If $z_j\in S(\theta_j)$ ($j=1,2$) and $\theta_1+\theta_2\leq \pi$, then $z_1z_2\in S(\theta_1+\theta_2)$. 	
		\item If $z_j\in S(\theta)$ ($j=1,2$) and $\theta\leq \frac{\pi}{2}$, then $z_1+z_2\in S(\theta)$. 
	\end{itemize}
	
	Computing explicitly, we see that \[
	P_1(z)=\delta_1-z, \quad P_2(z)=(\delta_1-z)(\delta_2-z)-\mu_1\nu_1. 
	\]
	Using the properties above and the assumption $\mu_1\nu_1\leq 0$, we see that $P_1(T)\subset S(\theta_1)$ and $P_2(T)\subset S(\theta_1+\theta_2)$, that is, \eqref{claim} for $k=1$ and $2$. Expanding the determinant defining $P_{k+2}$ along the last row, we find the recurrence formula\[
	P_{k+2}(z)=(\lambda_{k+2}-z)P_{k+1}(z)-\mu_{k+1}\nu_{k+1}P_k(z).
	\] 
	Assuming that \eqref{claim} has been proved for $k$ and $k+1$ and using again the elementary properties above, we see that if $z\in T$, then $P_{k+2}(z)$ is the sum of an element of $S(\theta_1+\cdots+\theta_{k+2})$ and an element that is either zero (if $\mu_{k+1}\nu_{k+1}=0$) or in $S(\theta_1+\cdots+\theta_k)$. Notice how we used the assumption on the off-diagonal elements of $A$ and \eqref{sumtheta}. The induction is now complete. 
\end{proof}

The proof can now be easily concluded. By Lemma \ref{lem:jacobi_spectrum}, every eigenvalue $E$ is the limit of a sequence $\{E_N\}_{N\geq 1}$, where each $E_N$ is an eigenvalue of $\mathcal{I}_N(a)$. By Lemma \ref{lem:tridiag_matrix} applied to $\mathcal{I}_N(a)$ and $\delta=\frac{3}{\pi}$ (all that really matters is that this quantity is uniform in $N$), we see that $E_N\notin  T$, where $T$ is as in the statement of Lemma \ref{lem:tridiag_matrix}. Since $T$ is open, we must necessarily have $E\geq1$. The proof of Theorem \ref{thm:1} is now complete.

\end{document}